\documentclass[letterpaper, 10 pt, conference]{ieeeconf}  

\IEEEoverridecommandlockouts                              

\usepackage{epsfig} 
\usepackage{epstopdf}
\usepackage{amsmath} 
\usepackage{amssymb} 

\newtheorem{theorem}{Theorem}

\usepackage{algorithm,algorithmic}

\usepackage{epstopdf}

\DeclareGraphicsExtensions{.pdf,.jpeg,.png,.eps}
\usepackage[utf8]{inputenc} 
\usepackage[T1]{fontenc}    
\usepackage{doi,url}        
\usepackage{booktabs}       
\usepackage{amsfonts}       
\usepackage{nicefrac}       
\usepackage{microtype}      

\hypersetup{%
linkcolor=blue,anchorcolor=blue,%
citecolor=blue,filecolor=blue,%
menucolor=blue,%
urlcolor=blue,%
colorlinks=true}

\title{\LARGE \bf
Least squares dynamics in Newton-Krylov Model Predictive Control
}

\author{Andrew Knyazev$^{1}$ and Alexander Malyshev$^{2}$
\thanks{*The work of the second author has been supported by MERL}
\thanks{$^{1}$Andrew Knyazev is with Mitsubishi Electric Research Laboratories (MERL)
201 Broadway, 8th floor, Cambridge, MA 02139, USA.
        {\tt\small knyazev@merl.com}, {\tt\small \url{www.merl.com/people/knyazev}}%
}
\thanks{$^{2}$Alexander Malyshev is with University of Bergen,
Department of Mathematics, PB 7803, 5020 Bergen, Norway \newline
        {\tt\small alexander.malyshev@uib.no}}%
}

\begin{document}

\maketitle
\thispagestyle{empty}
\pagestyle{empty}

\begin{abstract}
Newton-Krylov methods for nonlinear Model Predictive Control are pioneered by Ohtsuka under the name ``C/GMRES''.
Ohtsuka eliminates a system state over the horizon from Karush-Kuhn-Tucker stationarity conditions of a Lagrangian using equations of system dynamics. We propose instead using least squares to fit the state to the dynamics and some constraints on the state, if they are inconsistent. Correspondingly modified Newton-Krylov methods are described.
Numerical tests demonstrate workability of our modification.
\end{abstract}

\section{Introduction}
\label{sec:intro}

The paper is concerned with Model Predictive Control (MPC), see, e.g.,\ \cite{CaBo:04,RaMa:09,GrPa:11,DiFeHa:09},
for cases, where a model of a state of a system dynamics contradicts to some state constrains, making MPC infeasible. The contradictions may appear, e.g., from uncertainties and inaccuracies.
We propose using least squares to fit the state to both the dynamics and the contradicting constraints, within a framework of Newton-Krylov methods for nonlinear MPC (NMPC), pioneered by Ohtsuka \cite{Oht:04} for numerical solution of the MPC problems.

As an illustrating example (also used in our numerical tests), let us consider a continuous dynamical system where the state automatically satisfies an equality constraint, e.g., the state is on a smooth manifold, e.g.,\ a sphere in \cite{KnMa:15}. Discretized dynamical models, used for state prediction over a finite MPC horizon, approximate the continuous model and may not exactly satisfy the state equality constraint of the continuous case.
When the predictive horizon is long, the predicted trajectories may deviate far from the manifold determined by the equality constraint.
In \cite{KnMa:15}, we propose solving optimal control problems over smooth manifolds by using
the so-called ``structure preserving integration methods'' \cite{HaLuWa:06} within the Ohtsuka's method \cite{Oht:04,ShOhDi:09,KnFuMa:15}.

In the present work, we introduce a new prediction technology, aimed at removing the inconsistency of the state dynamics with some equality constraint on the state, by means of the least squares.
In our inner-outer approach, the inner layer is the  least squares fit of the state to the dynamics and the inconsistent constraints on the state, while the outer layer is NMPC solved by the Newton-Krylov methods of Ohtsuka.

We formulate a theoretical framework of two-level MPC, develop a numerical method
similar to Ohtsuka's method, and show numerical results for a test minimum-time problem describing motion on a unit sphere with constrained controls.

\section{Least squares dynamics in continuous MPC}
\label{sec:predict}

MPC determines a control input $u(t)$ by solving a prediction model on a finite horizon $[t,t+T]$. We consider a modified, using unknown disturbance vectors $\eta_f$ and $\eta_g$, variant of the prediction model from \cite{KnMa:15}, where the control $u(\tau)$ and a parameter vector $p$ minimize
\begin{equation}\label{min}
\min_{u,p} J(u,p),
\end{equation}
the performance index
\[
J(u,p) = \phi(x(\tau),p)|_{\tau=t+T}+\int_t^{t+T}L(\tau,x(\tau),u(\tau),p)d\tau
\]
subject to uncertain model dynamics
\begin{equation}\label{e1}
\frac{dx}{d\tau}=f(\tau,x(\tau),u(\tau),p)+\eta_f, \quad \tau\in[t,t+T],
\end{equation}
uncertain constraint on the state
\begin{equation}\label{e1s}
g(\tau,x(\tau),u(\tau),p)+\eta_g=0, \quad \tau\in[t,t+T],
\end{equation}
and the following certain constraints
\begin{equation}\label{e20}
x(\tau)|_{\tau=t} = x(t),
\end{equation}
\begin{equation}\label{e2}
C(\tau,x(\tau),u(\tau),p)|_{\tau\in[t,t+T]} = 0,
\end{equation}
\begin{equation}\label{e3}
\psi(x(\tau),p)|_{\tau=t+T} = 0.
\end{equation}
The initial value $x(\tau)|_{\tau=t}$ for the time-dependent differential equation (\ref{e1})
is the current state vector $x(t)$ of the dynamic system. The control vector $u=u(\tau)$,
which solves the prediction problem, is used as an input to control the dynamic system at time $t$.
The components of the vector $p(t)$ are parameters of the system.

The generally nonlinear equation~(\ref{e1}) exactly describes the model system dynamics,
while the generally nonlinear constraint \eqref{e1s} is also exact. But the
disturbance vectors $\eta_f$ and $\eta_g$ are unknown and, if dropped, may result in inconsistency for
arbitrary $u$ and $p$, thus leading to an infeasible MPC problem.
Assuming the vector-functions $u$ and $p$ fixed, the disturbance vectors $\eta_f$ and $\eta_g$ can be minimized
with respect to the function $x$ over the horizon via least squares, i.e.
\begin{equation}\label{LS}
\min_{x} S(x),
\end{equation}
where
\begin{align*}
S(x)=\left\|f(\tau,x(\tau),u(\tau),p)-dx/d\tau\right\|^2_f\\
+\left\|g(\tau,x(\tau),u(\tau),p)\right\|^2_g,
\end{align*}
and $\|\cdot\|_f$ and $\|\cdot\|_g$ are functional norms, e.g., based on the weighted $L_2$ norm of a function $h$
as $\|h\|_{W^{-1}}^2=\int h^T(\tau)W^{-1}(\tau)h(\tau)d\tau$ with the weight matrix $W^{-1}(\tau)$.
We note that the solution $x(\tau)$ over the horizon $\tau\in[t,t+T]$ has the given fixed initial, when $\tau=t$, value $x(t)$.

\section{Relaxed dynamics alternating minimization}
\label{sec:am}
Our discussion in \S \ref{sec:predict} motivates relaxing \eqref{e1} and \eqref{e1s} by simply adding the term $S(x)$ minimized in \eqref{LS} to the performance index $J(u,p)$ to be minimized in \eqref{min}, i.e.
\begin{equation}\label{minx}
\min_{u,p,x} J(u,p)+S(x),
\end{equation}
subject to only the certain constraints, i.e. \eqref{e20}, \eqref{e2}, and \eqref{e3}.

Explicitly adding $x$ to the set of minimization variables in \eqref{minx}, may add computation costs
to perform minimization, compared to the original setup \eqref{min}.
A well known idea of alternating minimization, see, e.g.,\ \cite{pu2014fast}, may reduce computations
by iteratively minimizing $J(u,p)+S(x)$ alternatively and separately with respect to $u,p$ and with respect to $x$.

One can interpret such an alternating minimization as inner-outer approach, where the inner layer is the  least squares fit \eqref{LS} of the state to the dynamics and the inconsistent constraints on the state, while the outer layer is NMPC minimization \eqref{min}, solved iteratively. Newton-Krylov methods of Ohtsuka \cite{Oht:04,KnFuMa:15} are examples of interest of iterative minimization \eqref{min} of the performance index. In the next section, we describe in the discrete case, how the original setup from \cite{Oht:04}, where \eqref{e1} and \eqref{e1s} are treated as exact certain constraints, can be modified to substitute the least squares minimization \eqref{LS} for \eqref{e1} and \eqref{e1s},
formulating the discrete Karush-Kuhn-Tucker (KKT) necessary conditions of \eqref{min} with the relaxed dynamics.

\section{Least squares discrete dynamics in KKT}
\label{sec:nk}
Continuous formulation of the finite horizon prediction problem stated above can be discretized
on a uniform time grid over the horizon $[t,t+T]$ partitioned into $N$ equal time steps of size $\Delta\tau$, and the time-continuous vector functions $x(\tau)$ and $u(\tau)$ are sampled at the grid points $\tau_i$, $i=0,1,\ldots,N$ and denoted by the indexed values $x_i$ and $u_i$ respectively. The integral of the performance cost $J$ over the horizon is approximated by means of the rectangular quadrature rule. The time derivative of the state vector is approximated by the forward difference formula.

Before deriving the Euler equations for the NMPC formulation, we discretize $x$ in the least squares minimization~\eqref{LS},
\begin{align*}
\min_{x_1,x_2,\ldots x_N} \sum_{i=0}^{N-1}\|(x_{i+1}-x_{i})/\Delta\tau-
f(\tau_i,x_i,u_i,p)\|^2_{W_f^{-1}}\\
+\|g(\tau_{i+1},x_{i+1},u_{i+1},p)\|^2_{W_g^{-1}},
\end{align*}
keeping the first component $x_0=x(t)$ fixed, where $\|\cdot\|_{W^{-1}}$ denote weighted,
using a matrix $W$, $2$-norms of vectors.

When the disturbances $\eta_f$ and $\eta_g$ are of random nature, the covariance matrices $W_f$ and $W_g$ may be available.
In our test examples in \S \ref{sec:ex}, we use the covariance matrices of the form $W_f=\alpha^{-1}I$ and $W_g=\beta^{-1}I$ with $\alpha=1$ and a suitable scalar $\beta>0$, with $I$ being the identity matrix.

For convenience, we introduce the block bidiagonal matrix
\[
B = \frac{1}{\Delta\tau}\begin{pmatrix}I\\-I&I\\&\ddots&\ddots\\&&-I&I\end{pmatrix}
\]
and the vectors
\[
R=B\begin{bmatrix}x_1\\x_2\\\vdots\\x_N\end{bmatrix}-
\begin{bmatrix}f(\tau_0,x_0,u_0,p)+x_0/\Delta\tau\\
f(\tau_1,x_1,u_1,p)\\\vdots\\f(\tau_{N-1},x_{N-1},u_{N-1},p)\end{bmatrix},
\]
\[
G = \begin{bmatrix}g(\tau_1,x_1,u_1,p)\\g(\tau_2,x_2,u_2,p)\\\vdots
\\g(\tau_{N},x_{N},u_{N},p)\end{bmatrix}.
\]
In this notation, the discrete version of the least squares minimization \eqref{LS} takes the following form,
\[
\min_xR^TW_f^{-1}R+GW_g^{-1}G.
\]
The gradients with respect to $x$ of the vectors $G$  and $R$ equal
\[
\nabla G=\begin{bmatrix}\nabla_xg(\tau_1,x_1,u_1,p)\\&\hspace{-5em}\nabla_xg(\tau_2,x_2,u_2,p)\\&\ddots
\\&&\hspace{-4em}\nabla_xg(\tau_{N},x_{N},u_{N},p)\end{bmatrix},
\]
\[
\nabla R=B-\begin{bmatrix}0\\\nabla_xf(\tau_1,x_1,u_1,p)&0\\&\hspace{-5em}\nabla_x f(\tau_2,x_2,u_2,p)&0\\
&\ddots&&\hspace{-2em}\ddots\\&&\hspace{-7em}\nabla_xf(\tau_{N-1},x_{N-1},u_{N-1},p)&0\end{bmatrix}.
\]
Hence the solution $x_i$, $i=1,\ldots N$, of the discrete least squares minimization satisfies the equation
\begin{equation}\label{e4}
(\nabla R)^TW_f^{-1}R+(\nabla G)^TW_g^{-1}G=0.
\end{equation}

The discretized optimal control problem NMPC is then formulated as follows:
\[
\min_{u_i,p}\left[\phi(x_N,p) + \sum_{i=0}^{N-1}L(\tau_i,x_i,u_i,p)\Delta\tau\right],
\]
subject to the system (\ref{e4}) for $x_i$ and the equality constraints
\begin{equation}\label{e5}
C(\tau_i,x_i,u_i,p) = 0,\quad  i = 0,1,\ldots,N-1,
\end{equation}
\begin{equation}\label{e6}
\psi(x_N,p) = 0.
\end{equation}

Necessary optimality conditions for the discretized finite horizon problem
can be derived by means of the discrete Lagrangian function
\begin{eqnarray*}
&&\mathcal{L}(X,U)=\phi(x_N,p)+\sum_{i=0}^{N-1}
L(\tau_i,x_i,u_i,p)\Delta\tau\\
&&+\,\lambda_0^T[x_0-x(t)]\\[2ex]
&&+[\lambda_{1}^T\ldots\lambda^T_{N}][(\nabla R)^TW_f^{-1}R+(\nabla G)^TW_g^{-1}G]\Delta\tau\\
&&+\sum_{i=0}^{N-1}\mu_i^TC(\tau_i,x_i,u_i,p)\Delta\tau+\nu^T\psi(x_N,p),
\end{eqnarray*}
where we gather the variables into vectors $X = [x_i\; \lambda_i]^T$, $i=0,1,\ldots,N$, and
$U = [u_i\; \mu_i\; \nu\; p]^T$, $i=0,1,\ldots,N-1$.
Here, $\lambda=[\lambda_{1}^T\ldots\lambda^T_{N}]^T$ is the costate vector, and
$\mu$ is the Lagrange multiplier vector associated with the constraint~(\ref{e5}).
The terminal constraint (\ref{e6}) is relaxed by the aid of the Lagrange multiplier $\nu$.

Calculating the derivatives of the Lagrangian $\mathcal{L}$ we obtain the necessary
optimality KKT conditions,
$\mathcal{L}_{\lambda_i}=0$, $\mathcal{L}_{x_i}=0$, $i=0,1,\ldots,N$,
$\mathcal{L}_{u_j}=0$, $\mathcal{L}_{\mu_j}=0$, $i=0,1,\ldots,N-1$,
$\mathcal{L}_{\nu_k}=0$, $\mathcal{L}_{p_l}=0$.

We further convert the KKT conditions into a nonlinear equation $F[U,x,t]=0$,
where the vector $U$ combines the control input $u$, the Lagrange multiplier
$\mu$, the Lagrange multiplier $\nu$, and the parameter $p$, all in one vector:
\[
U(t)=[u_0^T,\ldots,u_{N-1}^T,\mu_0^T,\ldots,\mu_{N-1}^T,\nu^T,p^T]^T.
\]
The vector argument $x$ in $F[U,x,t]$ denotes the current measured or estimated
state vector, which serves as the initial vector $x_0$ in the following procedure,
which eliminates the state variables $x_i$ and costate variables $\lambda_i$.
\begin{enumerate}
\item Having the current state $x_0$, measured or estimated, we compute
$x_i$, $i=1,2\ldots,N$, by solving least squares equations (\ref{e4}) instead of the forward Euler method
$x_{i+1}=x_i+f(\tau_i,x_i,u_i,p)\Delta\tau$ of \cite{Oht:04}.

Then compute the costates $\lambda_i$, $i=N,N\!-\!1,\ldots,1$, from the system of linear equations
\[
\frac{\partial\cal L}{\partial x}(X,U)=0.
\]
The value $\lambda_N$ is defined by the differentiation of the term $\nu^T\psi(x_N,p)$
with respect $x$.

\item Calculate $F[U,x,t]$, using just obtained $x_i$ and $\lambda_i$, as
\begin{eqnarray*}
F[U,x,t]=\left[\begin{array}{c}\begin{array}{c}
\frac{\partial\cal L}{\partial u_0}(X,U)\\
\vdots\\\frac{\partial\cal L}{\partial u_i}(X,U)\\
\vdots\\\frac{\partial\cal L}{\partial u_{N-1}}(X,U)\end{array}\\\;\\
\begin{array}{c}C(\tau_0,x_0,u_0,p)\Delta\tau\\
\vdots\\C(\tau_i,x_i,u_i,p)\Delta\tau\\\vdots\\
C(\tau_{N-1},x_{N-1},u_{N-1},p)\Delta\tau\end{array}\\\;\\
\psi(x_N,p)\\[2ex]
\frac{\partial\cal L}{\partial p}(X,U)
\end{array}\right].
\end{eqnarray*}
\end{enumerate}

The equation with respect to the unknown vector $U(t)$
\begin{equation}\label{e7}
 F[U(t),x(t),t]=0
\end{equation}
gives the required necessary optimality conditions.

Ohtsuka in \cite{Oht:04} proposes solving \eqref{e7} using Newton-Krylov methods
applied a forward-difference approximation to the Jacobian $F_U$ as described in
\cite{Kel:95}. In the next section, we repeat the necessary details, following
\cite{Oht:04,KnFuMa:15,7526060},
only slightly modified to take into account the relaxed dynamics.

\section{Newton-Krylov methods to solve KKT}
\label{sec:algo}

Let us assume that the dynamic system, which is controlled with the MPC approach, is sampled on a uniform time grid $t_j=j\Delta t$, $j=0,1,\ldots$ and denote $x_j=x(t_j)$. Equation (\ref{e7}) must be solved at each time step $t_j$ online on the controller board, which is the most computationally challenging part of an NMPC implementation for systems with fast dynamics.

The nonlinear equation $F[U_j,x_j,t_j]=0$ with respect to the unknown variables $U_j$ approximating $U(t_j)$
is equivalent to the following equation
\[
F[U_j,x_j,t_j]-F[U_{j-1},x_j,t_j]=b_j,
\]
where
\begin{equation}\label{e8}
b_j=-F[U_{j-1},x_j,t_j].
\end{equation}

Using a small scalar $h>0$, which is, in general, different from the time steps $\Delta t$ and $\Delta\tau$, we introduce, as, e.g., in \cite{Kel:95}, the forward difference operator
\begin{eqnarray}\label{e9}
a_j(V)=(F[U_{j-1}+hV,x_j,t_j]-F[U_{j-1},x_j,t_j])/h
\end{eqnarray}
approximating the derivative $F_U[U_{j-1},x_j,t_j](V)$ along the direction $V$.
We remark that the equation $F[U_j,x_j,t_j]=0$ is equivalent to the operator equation
$a_j(\Delta U_j/h)=b_j/h$, where $\Delta U_j=U_j-U_{j-1}$.

Let us introduce an $m\times m$ matrix $A_j$ with the columns $A_je_k$, $k=1,\ldots,m$,
defined by the formula $A_je_k=a_j(e_k)$, where $m$ is the dimension of the vector $U$
and $e_k$ denotes the $k$-th column of the $m\times m$ identity matrix.
The matrix $A_j$ is an $O(h)$ approximation of the Jacobian matrix $F_U[U_{j-1},x_j,t_j]$,
which is symmetric by Theorem~\ref{th1}.

\begin{theorem}\label{th1}
The Jacobian matrix $F_U[U,x,t]$ is symmetric.
\end{theorem}
\begin{proof}
The equation $\mathcal{L}_X(X,U)=0$ is solvable with respect to $X$
due to the solvability of the least squares minimization for $x_i$ and a system of linear equations for $\lambda_i$.
The rest of the proof is identical to that in \cite{7526060} for the case of the exact dynamics and is provided here for completeness.

Let us denote the solution to $\mathcal{L}_X(X,U)=0$ by $X=g(U)$. Then $F[U]=\mathcal{L}_U(g(U),U)$ and
\[
F_U=\mathcal{L}_{UU}(g(U),U)+\mathcal{L}_{UX}(g(U),U)g_U.
\]
Differentiation of the identity $\mathcal{L}_U(g(U),U)=0$ with respect to $U$
gives the identity
\[
\mathcal{L}_{UU}(g(U),U)+\mathcal{L}_{UX}(g(U),U)g_U(U)=0.
\]
Differentiation of the identity $\mathcal{L}_X(g(U),U)=0$ with respect to $U$
gives the identity
\[
\mathcal{L}_{XU}(g(U),U)+\mathcal{L}_{XX}(g(U),U)g_U(U)=0.
\]
Hence $g_U=-\mathcal{L}_{XX}^{-1}(g(U),U)\mathcal{L}_{XU}(g(U),U)$ and
\begin{align}
F_U[U] =& \mathcal{L}_{UU}(g(U),U)\label{e7a}\\
&{}-\mathcal{L}_{UX}(g(U),U)\mathcal{L}_{XX}^{-1}(g(U),U)\mathcal{L}_{XU}(g(U),U),\notag
\end{align}
which is called the Schur complement of the symmetric Hessian matrix of $\mathcal{L}$
at the point $(X,U)=(g(U),U)$. The Schur complement of any symmetric matrix is
symmetric.
\end{proof}

Suppose that an approximate solution $U_0$ to the equation $F[U_0,x_0,t_0]=0$ is
available. Finding sufficiently accurate approximation $U_0$ is crucial for success of
Newton-like methods and search for it is usually a challenging operation. However, we omit
descriptions of suitable methods for finding the starting value $U_0$ here
because it is unrelated to, although needed for,  the ``warm-start'' procedure described below.

The first block entry of $U_0$ is taken as the input control $u_0$ at the state $x_0$.
The next state $x_1=x(t_1)$ is measured by sensors or estimated.

At the time $t_j$, $j>1$, we have the state $x_j$ and the vector $U_{j-1}$
from the previous time $t_{j-1}$. Our goal is to solve the following equation
with respect to $V$:
\begin{equation}\label{e11}
 a_j(V)=b_j/h.
\end{equation}
Then we set $\Delta U_j=hV$, $U_j=U_{j-1}+\Delta U_j$ and choose the first block
component of $U_j$ as the control $u_j$. The next system state $x_{j+1}=x(t_{j+1})$
is measured by sensors or estimated.

A direct way to solve the operator equation (\ref{e11}) is forming the matrix $A_j$
explicitly and then solving the system of linear equations $A_j\Delta U_j=b_j$;
e.g., by the Gaussian elimination.

A faster alternative is solving (\ref{e11}) by Krylov iterative methods
(such as GMRES \cite{Oht:04,Kel:95}, or MINRES \cite{km15}, possibly with preconditioning \cite{7526060}),
where the operator $a_j(V)$ is used without explicit construction of the matrix $A_j$; cf., \cite{Oht:04,Kel:95}.
Krylov methods, applied to a finite difference approximation \eqref{e9} of a Jacobian, are
call ``Newton-Krylov methods'' in \cite{Kel:95}.

\section{Proof of concept numerical example}
\label{sec:ex}
We numerically simulate a minimum-time motion from an initial state $x_0$ to a terminal state $x_f$
over the unit two-dimensional sphere in $R^3$. The system dynamics is governed by the system of ordinary differential equations
\[
\dot{x}=\begin{bmatrix}0&0&\cos u\\0&0&\sin u\\-\cos u&-\sin u&0\end{bmatrix}x,
\]
where the control input $u$ is subject to the inequality constraint $|u-c|\leq r$, which we relax
with the equality constraint
\[
(u-c)^2+u_{d}^2-r^2=0.
\]
The variable $u_d$ is fictitious and controlled by the scalar $w_d$ introduced below.

The cost function is $J=p-\int_t^{t_f}w_du_d$, where $p=t_f-t$ is the time to destination,
and $w_d$ is a small positive constant.

We choose the receding horizon coinciding with the interval $[t,t_f]$. The horizon is parameterized by the dimensionless time $\tau\in[0,1]$ by means of the linear mapping $\tau\to t+\tau p$. The normalized interval $[0,1]$ is partitioned uniformly into the grid $\tau_i=i\Delta\tau$, $i=0,1,\ldots,N$, with the step size $\Delta\tau=1/N$.
The discretized variables include the state $x_i$ and costate $\lambda_i$, the control input $u_i$ and slack variable $u_{d,i}$, the Lagrange multipliers $\mu_i$ and $\nu$, the parameter $p$.

The uncertain predictive model of the dynamical system on the receding horizon is the forward Euler method
\begin{equation}\label{eEu}
\frac{x_{i+1}-x_i}{p\Delta\tau}=A(u_i)x_i,
\end{equation}
where
\[
A(u_i)=\begin{bmatrix}0&0&\cos u_i\\0&0&\sin u_i\\-\cos u_i&-\sin u_i&0\end{bmatrix}.
\]
The truncation error of the Euler methods is the disturbance $\eta_f$ in (\ref{e1}).
We remark that $\eta_f$ is not random here and highly correlated with the state function $x(\tau)$.

It is directly verified that the continuous system dynamics $\dot{x}=A(u)x$ satisfies the equality constraint on the state $x_{i}^Tx_i-1=0$, $i=1,\ldots,N$. Hence the constraint (\ref{e2}) has
$g(x_i)=x_i^Tx_i-1$ and $\eta_g=0$. The goal of the least squares minimization is to
satisfy the constraint (\ref{e2}) ``softly.'' We note that for this test problem it is
possible to satisfy the state constraint $x_{i}^Tx_i-1=0$ exactly by projecting $x_{i+1}$
onto the unit sphere after every step of \eqref{eEu}; see, e.g., \cite{KnMa:15}.

Yet another way in this example to satisfy the equality constraint $x_{i}^Tx_i-1=0$ is to use the
so-called exponential integrator $x_{j+1}=\exp\left(A(u_j)x_j\right)$, which preserves
the norm $\|x_j\|_2$. We use this exponential integrator for numerical simulation of
the system dynamics replacing measurements.

The discretized cost function is
\[
J=p\left(1-\Delta\tau w_d\sum_{i=0}^{N-1}u_{d_i}\right).
\]

We choose our least squares approximation of the state $x_i$, with the fixed initial value $x_0$
and a scalar parameter $\beta\geq0$,
\begin{equation}\label{e55}
\min_{x_i}\sum_{i=1}^{N}\|x_{i}-x_{i-1}-\Delta\tau pA(u_{i-1})x_{i-1}\|^2_2+
\beta^2|x_{i}^Tx_{i}-1|^2.
\end{equation}
The parameter $\beta$ determines the force of satisfying the equality constraint
$x_{i}^Tx_i-1=0$: the larger the constant $\beta$ the larger the enforcement.

The least squares minimization problem is equivalent to the system of nonlinear equations
\[
\left(B^TB+2\beta^2\begin{bmatrix}(x_1^Tx_1-1)I\\
&\ddots\\&&(x_N^Tx_N-1)I\end{bmatrix}\right)\!\!
\begin{bmatrix}x_1\\\vdots\\x_N\end{bmatrix}
\]
\[
-\begin{bmatrix}\left(\Delta\tau pA(u_0)+I\right)x_0\\
0\\\vdots\\0\end{bmatrix}=S(x,u,p)=0,
\]
where
\[
B=\begin{bmatrix}I\\-\Delta\tau pA(u_1)-I&I\\
&\ddots&\ddots\\&&-\Delta\tau pA(u_{N-1})-I&I\end{bmatrix}.
\]

The corresponding discrete Lagrangian function then has the following form
\begin{align*}
{\cal L}=&\;p\left(1-\Delta\tau w_d\sum_{i=0}^{N-1}u_{d,i}\right)+\lambda^TS(x,u,p)\\
&+\!\sum_{i=0}^{N-1}\mu_i^T[(u_i-c)^2+u_{d,i}^2-r^2]+\nu^T(x_N-x_f).
\end{align*}

The costate $\lambda$ satisfies the formula
\[
\lambda=\begin{bmatrix}\lambda_1\\\vdots\\\lambda_N\end{bmatrix}=(B^TB+2\beta^2D)^{-1}
\begin{bmatrix}0\\\vdots\\0\\-\nu\end{bmatrix},
\]
where $C$ is the block diagonal matrix given by
\[
D=\mbox{blockdiag}\{(x_i^Tx_i-1)I+2x_ix_i^T\}.
\]

The function $F(U,x_0,t)$, where
\begin{align*}
U=[u_0,\ldots,u_{N-1},u_{d,0},\ldots,u_{d,N-1},\\
\mu_0,\ldots,\mu_{N-1},\nu,p]^T,
\end{align*}
has the following rows from the top to bottom:
\begin{align*}
&2\begin{bmatrix}\mu_0(u_0-c)\\\vdots\\\mu_{N-1}(u_{N-1}-c)\end{bmatrix}-\Delta\tau p(B\lambda)^T
\begin{bmatrix}A'(u_0)x_0\\\vdots\\A'(u_{N-1})x_{N-1}\end{bmatrix}\\
&\hspace{9em}-\Delta\tau p\begin{bmatrix}0\\A'(u_1)\lambda_1\\\vdots\\A'(u_{N-1})\lambda_{N-1}\end{bmatrix}^T(Bx);\\
&2\begin{bmatrix}\mu_0u_{d,0}\\\vdots\\\mu_{N-1}u_{d,N-1}\end{bmatrix}-\Delta\tau pw_d\begin{bmatrix}1\\\vdots\\1\end{bmatrix};\\
&\begin{bmatrix}(u_0-c)^2+u_{d,i}^2-r^2\\\vdots\\(u_{N-1}-c)^2+u_{d,N-1}^2-r^2\end{bmatrix};\\
&x_N-x_f;\\
&1-\Delta\tau w_d\sum\limits^{N-1}_{i=0}u_{d,i}-\Delta\tau(B\lambda)^T
\begin{bmatrix}A(u_0)x_0\\\vdots\\A(u_{N-1})x_{N-1}\end{bmatrix}\\
&-\Delta\tau \begin{bmatrix}0\\A(u_1)\lambda_1\\\vdots\\A(u_{N-1})\lambda_{N-1}\end{bmatrix}^T(Bx).
\end{align*}

The example is chosen here for historical reasons---one of tests from our prior work \cite{KnMa:15}.
It~is not the most beneficial one illustrating effectiveness of the proposed least squares fit of the state with uncertain dynamics and constraints over the horizon, because in this example the state constraint to the sphere can in practice be actually certain, and satisfied with high accuracy by other means; e.g., \cite{KnMa:15}. The role of this example is a
\emph{proof of concept}.

\emph{Remark.} A very important circumstance arises in the problems with the state constraints derived from the system dynamics. The number of terminal constraints must be reduced to the dimension of the smooth manifold determined by the equality constraint on the state. In our case, the dimension of the sphere equals 2, and, therefore, the Lagrange multiplier $\nu$ must contain only 2 components instead of 3.
In our MATLAB implementation, we keep the components of $\nu$ corresponding to the $x$ and $y$ coordinates
of the terminal state, but the last component of the right-hand side in the equation for the costate is set to zero.
If the above described reduction of the terminal constraint is not fulfilled, then the subsequent computations lead to singular Jacobians in the Newton-type iterations.

\section{Numerical results}
\label{sec:numer}

We perform several preliminary numerical experiments in MATLAB with the test problem from \S~\ref{sec:ex}.
Problem~\eqref{e55} is solved by the MATLAB function \texttt{lsqnonlin} for nonlinear least squares problems. The operator equation (\ref{e11}) is solved by the \texttt{gmres} function of MATLAB.
The relative error tolerance for the GMRES iterations is $tol=10^{-5}$.
The number of grid points on the horizon is $N=10$, the sampling time
of the simulation is $\Delta t=1/200$, and $h=10^{-8}$.

Other constants are as follows: $c=0.5$, $r=0.1$, $w_{d}=0.005$, $\beta=10$.

The initial value for $U_0$ is computed by the MATLAB function \texttt{fsolve}, which finds
a solution to nonlinear equations by Newton-type methods. We note that finding good initial
approximation for \texttt{fsolve} may be non-trivial.

The trajectory satisfying the system dynamics $\dot{x}=A(u)x$ has been computed by the simple
exponential integrator $x_{j+1}=\exp\left(A(u_j)x_j\right)$ substituting the measurements.

Figure \ref{fig1} shows the computed trajectory $(x,y,z)$ on the sphere for the test example.
Figure \ref{fig2} left panel shows the $(x,y)$-projection of the computed trajectory.
Figure \ref{fig2} right panel shows the input control variable with the constraints.
Figure \ref{fig4} displays the number of GMRES iterations at the grid points.
Finally, Figure \ref{fig5} displays the 2-norm of the residual function $F[U]$ that is supposed to vanish.

Since our implementation is not optimized, we do not provide the timing. However, we note that
computations by the function \texttt{lsqnonlin} are relatively time consuming. We also observe numerically that
a successful execution of Newton-type iterations requires solution of the least squares minimization
problem with sufficiently high accuracy.

Comparison with other relevant numerical methods based on, e.g., multiple shooting \cite{ShOhDi:09}
is a subject for future research.

\section*{Conclusions}
A novel concept of least squares relaxation of state dynamics and some constraints
in NMPC calculations of control over the horizon is proposed via alternating minimization
and implemented in Newton-Krylov iterative methods. Numerical results for a proof of concept example demonstrate feasibility of computer implementations of the proposed technology.
Future research is needed to design numerically efficient algorithms
and to test our techniques for large uncertainties.

\begin{figure}
\noindent\centering{
\includegraphics[width=.95\linewidth]{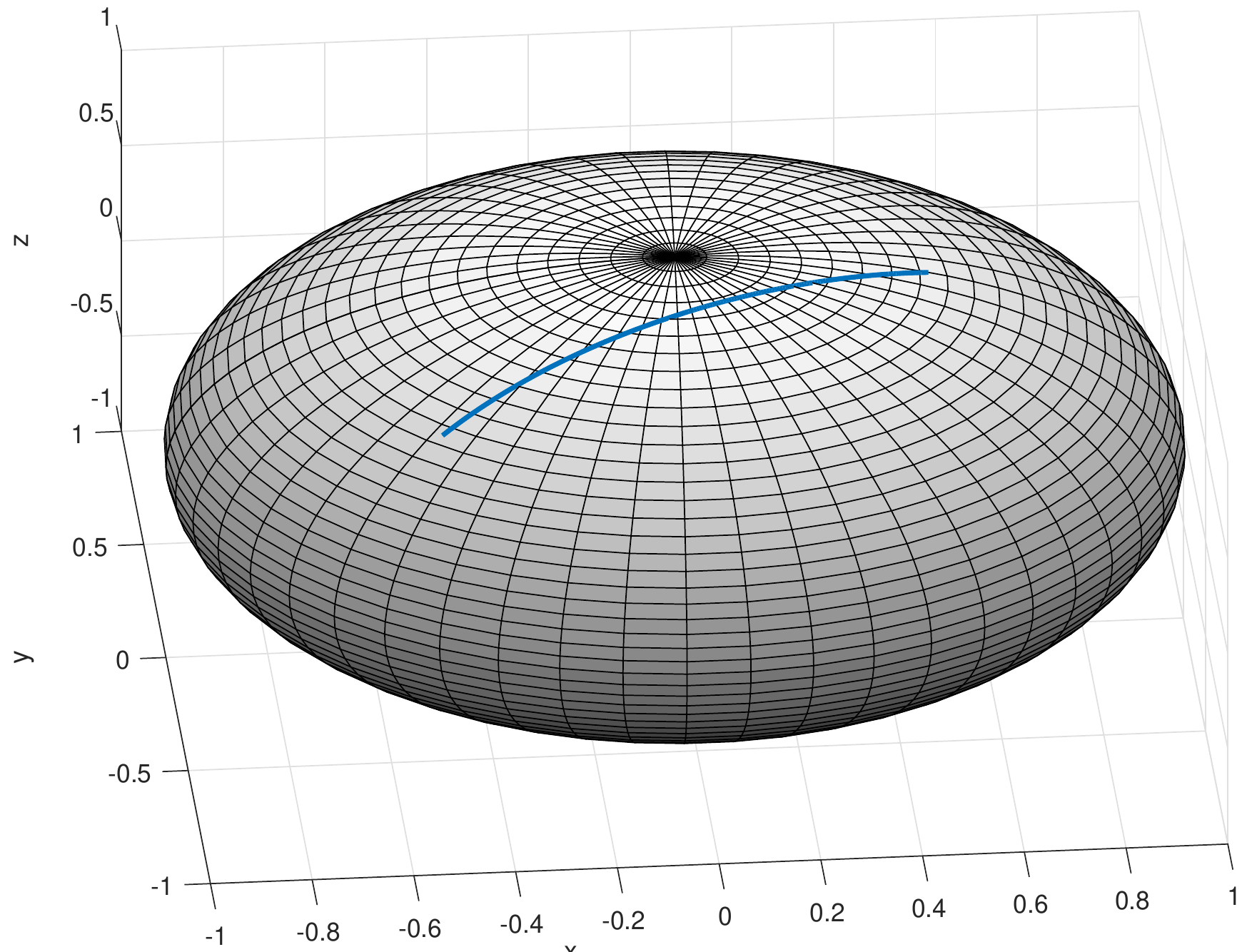}
}
\caption{Computed trajectory}
\label{fig1}
\end{figure}

\begin{figure}
\noindent\centering{
\includegraphics[width=.50\linewidth]{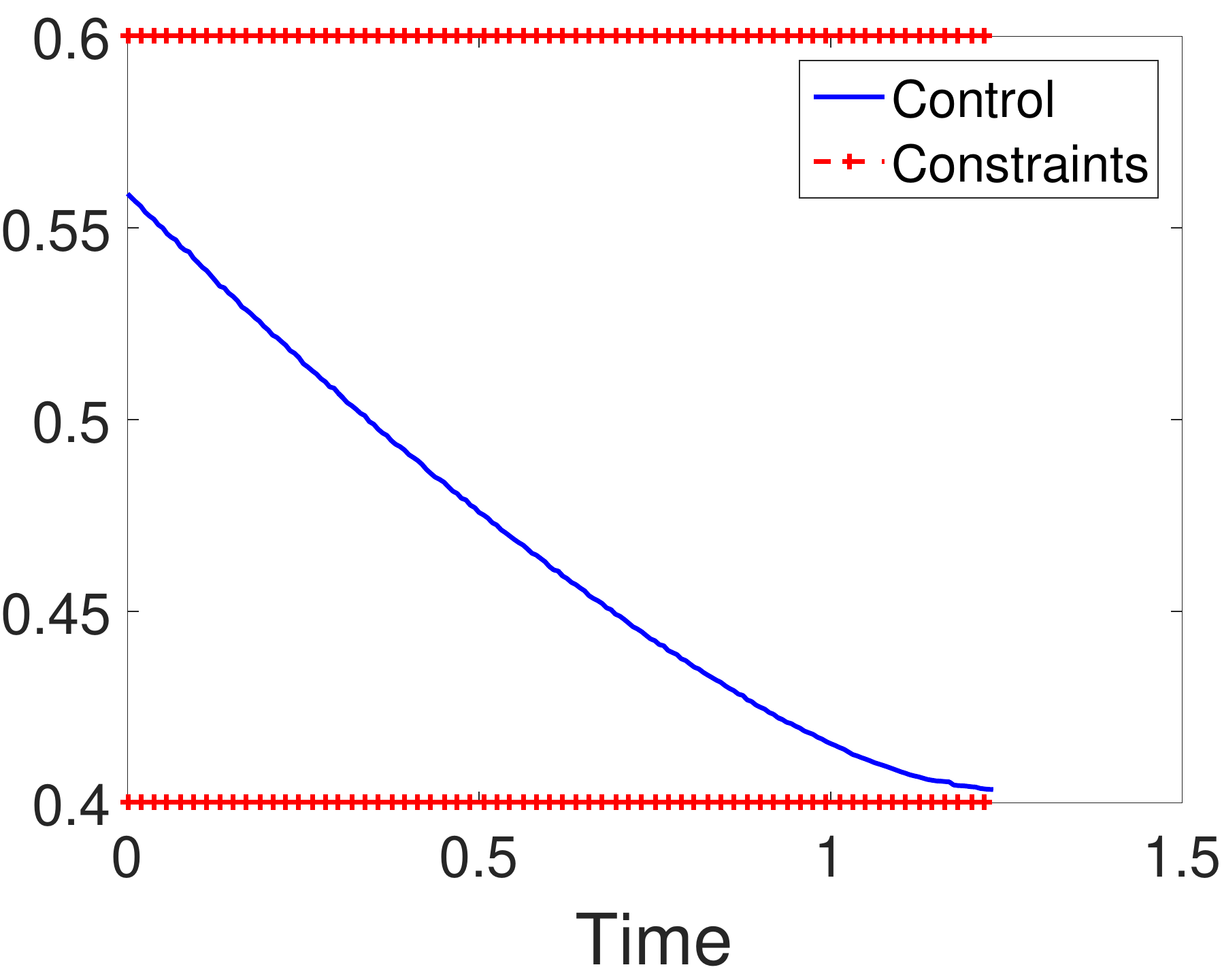}
\includegraphics[width=.46\linewidth]{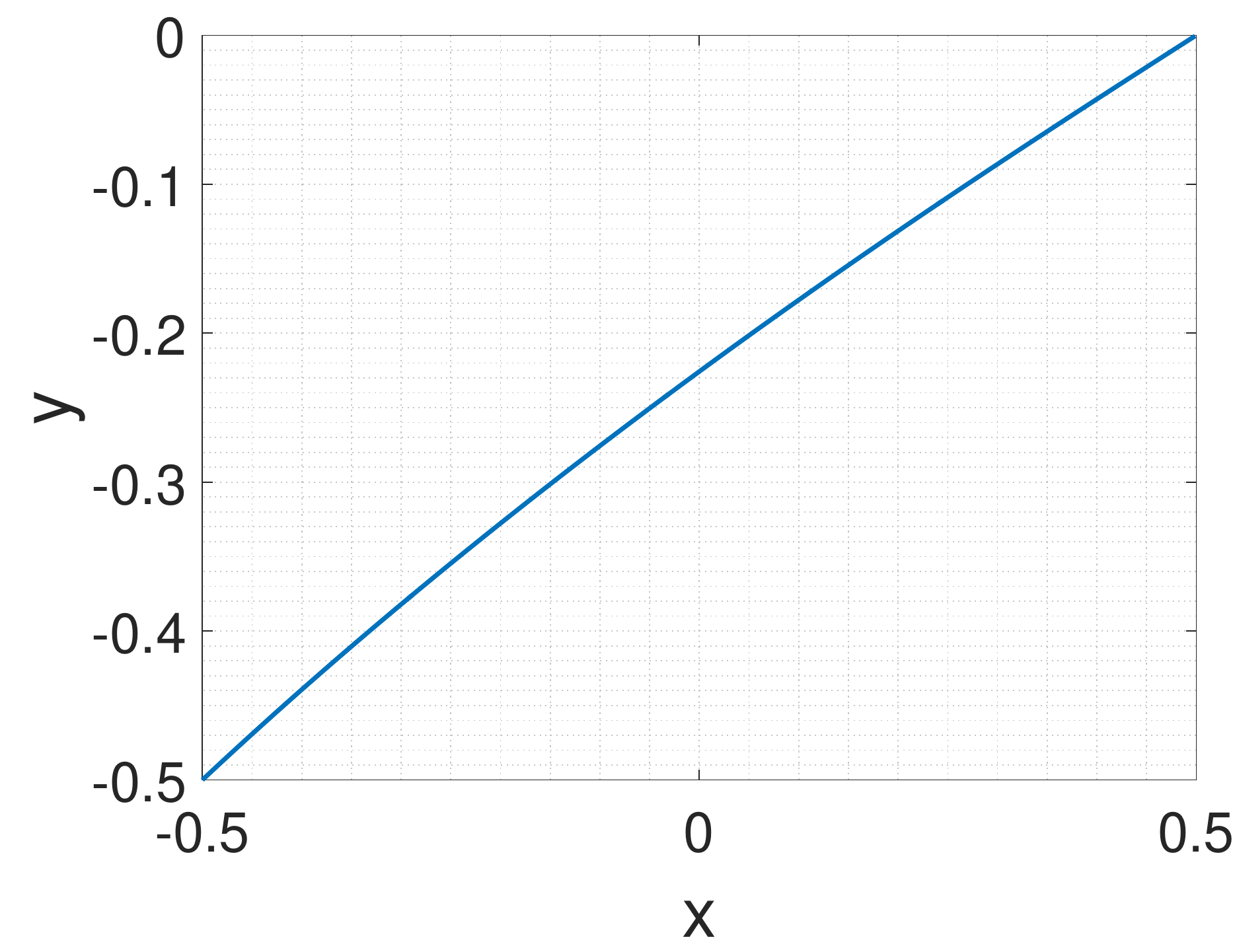}
}
\caption{$(x,y)$-projection of the trajectory (left) and control $u$ (right)}
\label{fig2}
\end{figure}

\begin{figure}
\noindent\centering{
\includegraphics[width=.95\linewidth]{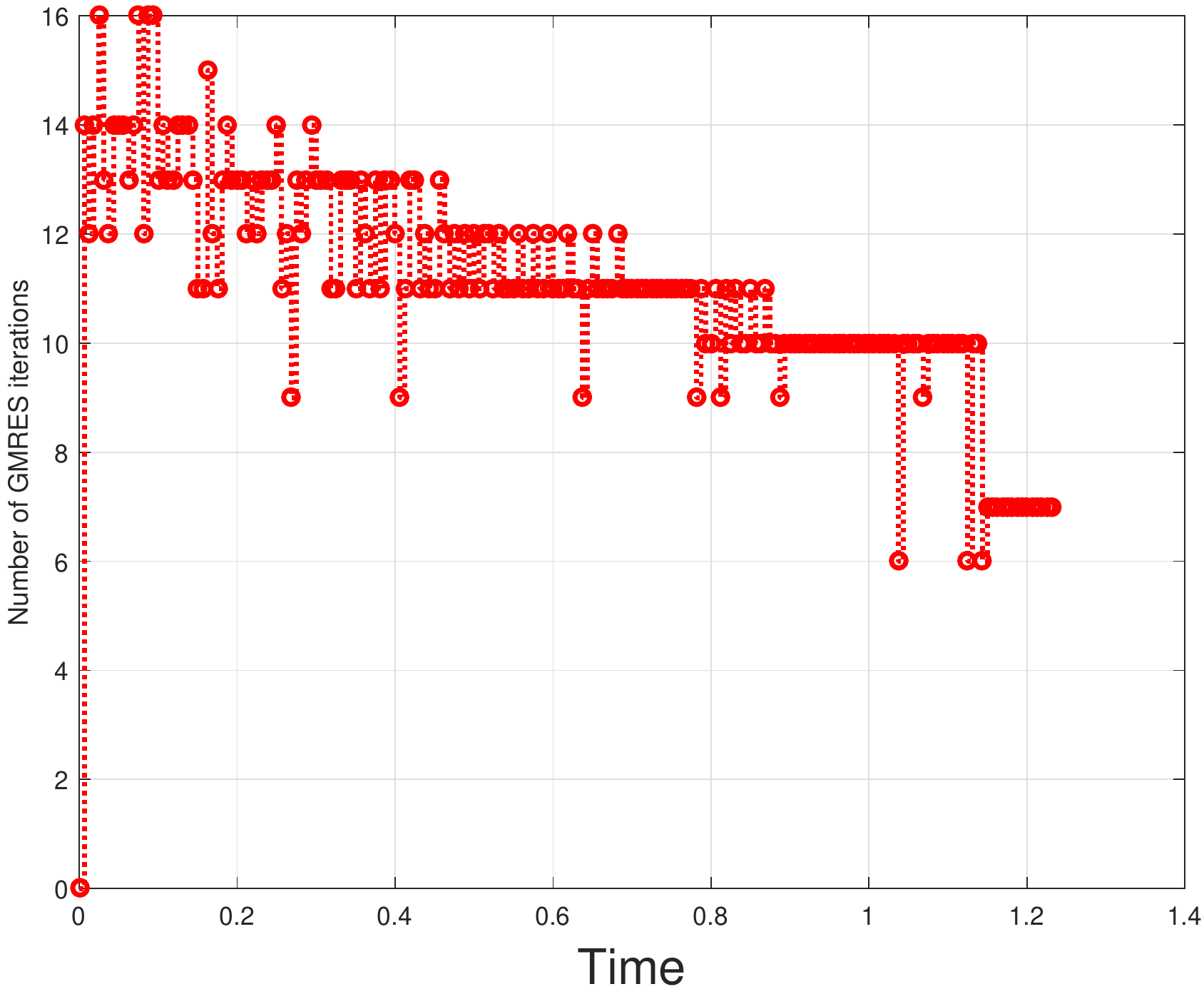}
}
\caption{The number of GMRES iterations}
\label{fig4}
\end{figure}

\begin{figure}
\noindent\centering{
\includegraphics[width=.95\linewidth]{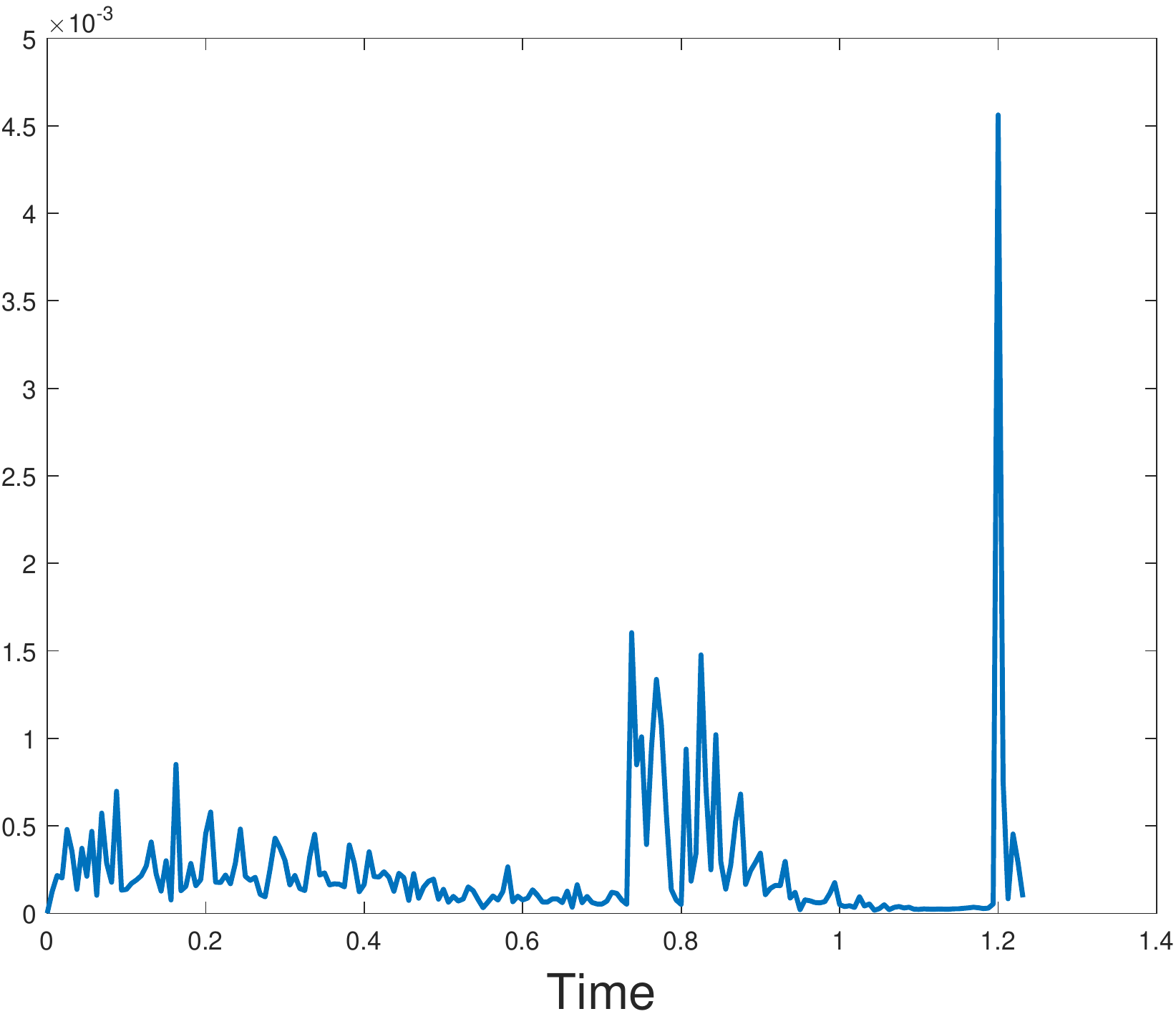}
}
\caption{2-norm of the residual $F[U]$}
\label{fig5}
\end{figure}

\bibliographystyle{IEEEbib}
\bibliography{refs}

\begin{thebibliography}{10}
\providecommand{\url}[1]{#1}
\csname url@rmstyle\endcsname
\providecommand{\newblock}{\relax}
\providecommand{\bibinfo}[2]{#2}
\providecommand\BIBentrySTDinterwordspacing{\spaceskip=0pt\relax}
\providecommand\BIBentryALTinterwordstretchfactor{4}
\providecommand\BIBentryALTinterwordspacing{\spaceskip=\fontdimen2\font plus
\BIBentryALTinterwordstretchfactor\fontdimen3\font minus
  \fontdimen4\font\relax}
\providecommand\BIBforeignlanguage[2]{{%
\expandafter\ifx\csname l@#1\endcsname\relax
\typeout{** WARNING: IEEEtran.bst: No hyphenation pattern has been}%
\typeout{** loaded for the language `#1'. Using the pattern for}%
\typeout{** the default language instead.}%
\else
\language=\csname l@#1\endcsname
\fi
#2}}

\bibitem{CaBo:04}
E.~F. Camacho and C.~Bordons, \emph{Model predictive control}, 2nd~ed.\hskip
  1em plus 0.5em minus 0.4em\relax Heidelberg, Germany: Springer, 2004,
  \doi{10.1007/978-0-85729-398-5}.

\bibitem{RaMa:09}
\BIBentryALTinterwordspacing
J.~B. Rawlings and D.~Q. Mayne, \emph{Model predictive control: Theory and
  design}.\hskip 1em plus 0.5em minus 0.4em\relax {LLC}: NobHill Publishing,
  2009. [Online]. Available: \url{http://jbrwww.che.wisc.edu/home/jbraw/mpc/}
\BIBentrySTDinterwordspacing

\bibitem{GrPa:11}
L.~Gr{\"u}ne and J.~Pannek, \emph{Nonlinear model predictive control. Theory
  and algorithms}.\hskip 1em plus 0.5em minus 0.4em\relax Springer, 2017,
  \doi{10.1007/978-3-319-46024-6}.

\bibitem{DiFeHa:09}
M.~Diehl, H.~J. Ferreau, and N.~Haverbeke, \emph{Efficient Numerical Methods
  for Nonlinear MPC and Moving Horizon Estimation}.\hskip 1em plus 0.5em minus
  0.4em\relax Berlin, Heidelberg: Springer Berlin Heidelberg, 2009, pp.
  391--417, \doi{10.1007/978-3-642-01094-1_32}.

\bibitem{Oht:04}
\BIBentryALTinterwordspacing
T.~Ohtsuka, ``A continuation/gmres method for fast computation of nonlinear
  receding horizon control,'' \emph{Automatica}, vol.~40, no.~4, pp. 563--574,
  2004. [Online]. Available:
  \url{http://www.sciencedirect.com/science/article/pii/S0005109803003637}
\BIBentrySTDinterwordspacing

\bibitem{KnMa:15}
A.~Knyazev and A.~Malyshev, ``Continuation model predictive control on smooth
  manifolds,'' in \emph{IFAC-PapersOnLine. 16th IFAC Workshop on Control
  Applications of Optimization CAO’2015 – Garmisch-Partenkirchen, Germany,
  6–9 October 2015}, vol.~48, no.~25, 2015, pp. 126--131,
  \doi{10.1016/j.ifacol.2015.11.071}.

\bibitem{HaLuWa:06}
E.~Hairer, C.~Lubich, and G.~Wanner, \emph{Geometric numerical integration.
  Structure preserving algorithms for ordinary differential equations},
  2nd~ed.\hskip 1em plus 0.5em minus 0.4em\relax Berlin Heidelberg:
  Springer-Verlag, 2006, \doi{10.1007/3-540-30666-8}.

\bibitem{ShOhDi:09}
Y.~Shimizu, T.~Ohtsuka, and M.~Diehl, ``A real-time algorithmfor nonlinear
  receding horizon control using multiple shooting and continuation/krylov
  method,'' \emph{Int. J. Robust Nonlin. Control}, vol.~19, no.~8, pp.
  919--936, 2009, \doi{10.1002/rnc.1363}.

\bibitem{KnFuMa:15}
\BIBentryALTinterwordspacing
A.~Knyazev, Y.~Fujii, and A.~Malyshev, ``Preconditioned continuation model
  predictive control,'' in \emph{2015 Proceedings of the Conference on Control
  and its Applications}, pp. 101--108. [Online]. Available:
  \url{http://epubs.siam.org/doi/abs/10.1137/1.9781611974072.15}
\BIBentrySTDinterwordspacing

\bibitem{pu2014fast}
Y.~Pu, M.~N. Zeilinger, and C.~N. Jones, ``Fast alternating minimization
  algorithm for model predictive control,'' \emph{IFAC Proceedings Volumes},
  vol.~47, no.~3, pp. 11\,980--11\,986, 2014,
  \doi{10.3182/20140824-6-ZA-1003.01432}.

\bibitem{Kel:95}
\BIBentryALTinterwordspacing
C.~Kelley, \emph{Iterative Methods for Linear and Nonlinear Equations}.\hskip
  1em plus 0.5em minus 0.4em\relax Society for Industrial and Applied
  Mathematics (SIAM), 1995, \doi{10.1137/1.9781611970944}. [Online]. Available:
  \url{http://epubs.siam.org/doi/abs/10.1137/1.9781611970944}
\BIBentrySTDinterwordspacing

\bibitem{7526060}
A.~Knyazev and A.~Malyshev, ``Sparse preconditioning for model predictive
  control,'' in \emph{2016 American Control Conference (ACC)}, July 2016, pp.
  4494--4499, \doi{10.1109/ACC.2016.7526060}.

\bibitem{km15}
------, ``Preconditioning for continuation model predictive control,'' in
  \emph{IFAC-PapersOnLine, 5th IFAC Conference on Nonlinear Model Predictive
  Control NMPC 2015, Seville, Spain, 17--20 September 2015}, vol.~48.\hskip 1em
  plus 0.5em minus 0.4em\relax Elsevier, 2015, pp. 191--196,
  \doi{10.1016/j.ifacol.2015.11.282}.

\end{thebibliography}

\end{document}